\documentclass{article}
\usepackage{spconf,amsmath,graphicx, shortcuts_OPT}
\usepackage[citecolor=blue, colorlinks=true]{hyperref}
\usepackage{amsmath}
\usepackage{amssymb}
\usepackage{amsthm}
\usepackage{algorithm}
\usepackage[noend]{algpseudocode}
\usepackage{subcaption}
\usepackage{bm}
\usepackage{shortcuts_OPT}
\usepackage{enumitem}

\usepackage{cite}

\theoremstyle{plain}
\newtheorem{theorem}{Theorem}[section] 
\newtheorem{lemma}[theorem]{Lemma}     

\theoremstyle{definition}

\newtheorem{assumption}[theorem]{Assumption}

\ninept

\title{Decentralized Learning with Dynamically Refined Edge Weights: A Data-Dependent Framework}
\name{Rongxing Du, Hoi-To Wai \thanks{Emails:~\url{rxdu@se.cuhk.edu.hk}, \url{htwai@se.cuhk.edu.hk}. This research was supported in part by project \#MMT-p5-23 of the Shun Hing Institute of Advanced Engineering, CUHK.}}
\address{Department of System Engineering and Engineering Management \\ The Chinese University of Hong Kong}
\begin{document}
%
\maketitle

%
\begin{abstract}
This paper aims to accelerate decentralized optimization by strategically designing the edge weights used in the agent-to-agent message exchanges. We propose a Dynamic Directed Decentralized Gradient ({\algname}) framework and show that the proposed data-dependent framework is a practical alternative to the classical directed DGD (Di-DGD) algorithm for learning on directed graphs. To obtain a strategy for edge weights refinement, we derive a design function inspired by the cost-to-go function in a new convergence analysis for Di-DGD. This results in a data-dependent dynamical design for the edge weights. A fully decentralized version of {\algname} is developed such that each agent refines its communication strategy using only neighbor's information. Numerical experiments show that {\algname} accelerates convergence towards stationary solution by 30-40\% over Di-DGD, and learns edge weights that adapt to data similarity.
\end{abstract}
\section{Introduction}
\label{sec:intro}
Decentralized learning has become popular for handling recent challenges from data sciences as it provides a scalable and privacy-enhancing solution through collaborative learning \cite{nedic2020distributed, chang2020distributed, Survey_of_FL}. The paradigm has proven to be essential for applications in signal processing and machine learning such as wireless sensor networks \cite{Ramesh2021mobilesensornetwork}, array signal processing \cite{Chen2021arraysignal}, and large language model inference \cite{du2025meta}.

This paper is concerned with decentralized learning over $n$ agents, where the underlying communication is constrained by a directed graph $\mathcal{G = (V,E)}$ with ${\cal V} = \{1,\ldots,n\}$. Our goal is to tackle:
\begin{equation}
\label{P1} \textstyle
    \min_{\prm \in \mathbb{R}^d } ~F(\prm) := \frac{1}{n} \sum_{i=1}^{n} f_i(\prm),
\end{equation}
where the $i$th objective function $f_i: \mathbb{R}^d \to \mathbb{R}$ is continuously differentiable and represents the private data held by agent $i$. 

The study of decentralized optimization algorithm for \eqref{P1}, as pioneered by the decentralized gradient (DGD) method in \cite{nedic2009distributed}, has attracted considerable attention as these algorithms enable optimization of \eqref{P1} without a central server, thus enhancing resiliency and flexibility. Recent works have either extended the analysis of DGD with tighter bounds \cite{spectral_gap, koloskova2020unified, yuan2023removing, zeng2018nonconvex}, or extended DGD to work under various settings, including gradient tracking \cite{qu2017harnessing, shi2015extra}, directed graphs \cite{nedic2014distributed, ACC_paper, Xi-row, Push-Pull, assran2019stochastic}, time varying graphs \cite{yau2024fully, Achieving-geometric-convergence}, etc. 
Most of the above works require the algorithms to work with a weighted adjacency matrix that assign {\it fixed} weights to the (directed or undirected) edges, or the algorithms work with random edge weights from a {\it fixed} distribution. Importantly, the spectral property of such weighted adjacency matrix determines the efficacy of the algorithm.

This paper concentrates on refining the edge weights to improve the communication efficiency of decentralized learning. We attempt the less-studied research question -- \emph{should the edge weights be data-dependent to accelerate the convergence of decentralized learning?} Intuitively, the answer is positive as agent shall focus on communicating with neighbors who hold data that are different from his/her own. To this end, we propose a dynamic edge weight refinement strategy, called dynamic Di-DGD ({\algname}), that seeks to optimize the edge weights simultaneously with the decentralized algorithm, while taking clues from iterates of the latter. Note we focus on the directed graph setting that enables decentralized edge weight refinement. 

Our framework is developed from the directed DGD algorithm (Di-DGD) in \cite{ACC_paper}. The contributions are summarized as:
\begin{itemize}[leftmargin=*]
    \item We propose the {\algname} framework which dynamically refines the adjacency matrix's weights {\it simultaneously} with the Di-DGD iterations. The update strategy is justified from a design function inspired by the analysis on Di-DGD, which involves the latest edge weights, iterates, and data.
    \item As a side contribution, we established a new finite-time convergence analysis for Di-DGD that is tailor made for smooth (but possibly non-convex) objective function, such as \eqref{P1}. We note that this yields a more general bound than the results in \cite{ACC_paper}. 
    \item We show that {\algname} admits a fully decentralized implementation. In this implementation scheme, each agent independently performs coordinate-wise updates of the adjacency matrix weights, leveraging only information from local and 1-hop neighbor. 
\end{itemize}
Finally, we conduct numerical experiments to compare the performance of the proposed {\algname} with Di-DGD. We show that the proposed algorithm has improved efficacy over Di-DGD and finds meaningful edge weights in a stylized problem setup.

\vspace{0.2cm}
\noindent
{\bf Related Works.} 
Our work is closely related to recent efforts on simultaneously learning graph topology and decentralized optimization, yet the latter have various limitations. For instance, \cite{le2023refined} considered undirected graph and proposed to learn a set of sparsified edge weights using Frank-Wolfe algorithm, \cite{shah2024adaptive} considered undirected graph with a stochastic network pruning strategy while taking care of the interplay between network topology and data heterogeneity. These approaches rely on a centralized, offline pre-processing step for edge weight refinement. The latter is not considered in this paper where we seek a decentralized and dynamical refinement scheme.

\vspace{-.1cm}
\section{Preliminaries}\vspace{-.1cm}
This section introduces the decentralized gradient descent (DGD) method for directed graph (Di-DGD) \cite{ACC_paper} that forms the backbone of our strategic communication based algorithm. We also present a new convergence analysis for Di-DGD and highlight the challenges on how to achieve fast convergence with the algorithm. 

Consider the directed graph ${\cal G} = ( {\cal V}, {\cal E} )$ with $n$ agents represented by ${\cal V} = \{ 1, \ldots, n \}$ such that $(i,j) \in {\cal E}$ represents a directed edge from $i$ to $j$. The edge set also contains self loops such that $(i,i) \in {\cal E}$ for all $i \in {\cal V}$. We set ${\cal N}_i$ as the in-neighborhood set of $i$, where an edge to $i$ is present. The graph is endowed with a (possibly asymmetric) row-stochastic adjacency matrix ${\bm A} \in {\cal A}_{\cal G}$. We denote the feasible set of adjacency matrices 
\beq 
{\cal A}_{\cal G} = \{ {\bm A} \in \RR_+^{n \times n} : {\bm A} {\bf 1} = {\bf 1}, ~A_{ij} = 0,~(j,i) \notin {\cal E} \}. 
\eeq 
Consider the following standard assumption on ${\cal G}$, i.e.,
\begin{assumption} \label{assu:strong_connect}
    The graph ${\cal G}$ is strongly connected and aperiodic, and its associated adjacency matrix ${\bm A} \in {\cal A}_{\cal G}$ is irreducible.
\end{assumption}
\noindent Given a strongly connected and aperiodic ${\cal G}$, an irreducible matrix ${\bm A} \in {\cal A}_{\cal G}$ can be constructed by setting $A_{ij} = 1 / ( |{\cal N}_i| + 1 )$ for any $(j,i) \in {\cal E}$.
Assumption \ref{assu:strong_connect} implies that there exists a left Perron eigenvector, $\bm{\pi}_{\bm A} \in \RR_{++}^n$, that satisfies $\bm{\pi}_{\bm A}^\top {\bm A} = \bm{\pi}_{\bm A}^\top$ and $\bm{\pi}^\top {\bf 1} = 1$ \cite{Horn_Johnson_1985}. In particular, we observe the limit $ \lim_{k \to \infty} ( {\bm A} )^k = {\bf 1} \bm{\pi}_{\bm A}^\top$ and there exists a spectral gap constant $\rho_{\bm A} \in (0,1]$ such that 
\beq \label{eq:spectral}
\lambda_{\max} ( {\bm A} - {\bf 1} \bm{\pi}_{\bm A}^\top ) \le 1 - \rho_{\bm A} < 1.
\eeq 
Note that if ${\bm A}$ is doubly stochastic (e.g., if ${\cal G}$ is undirected), then $\bm{\pi}_{\bm A} = (1/n) {\bf 1}$ and the limit evaluates to $\lim_{k \to \infty} ( {\bm A} )^k = (1/n) {\bf 1} {\bf 1}^\top$. In other words, the repeated actions of ${\bm A}$ will compute the exact average in the latter case.

The above discussions highlighted a key challenge for decentralized learning on directed graphs. Herein, performing average consensus with ${\bm A} \in {\cal A}_{\cal G}$ is non-trivial as the repeated action of ${\bm A}$ only results in a \emph{weighted average} as $\bm{\pi}_{\bm A} \neq (1/n) {\bf 1}$ in general. To this end, \cite{ACC_paper} proposed the Di-DGD algorithm: for any initialization $(\prm_i^0, {\bm y}_i^0)_{i=1}^n$ with ${\bm y}_i^0 > {\bm 0}$, for any $k \geq 0$,
\begin{subequations}
\label{strategy:1}
    \begin{align}
    \label{iter:prm}
    \prm_i^{k+1} & = \textstyle \sum_{j\in \mathcal{N}_i} A_{ij} \prm_j^k - \gamma_k ( y_{i,i}^k n )^{-1} \nabla f_i(\prm_i^k), \\
    \label{iter:y_k}
    {\bm y}_i^{k+1} & = \textstyle \sum_{j\in \mathcal{N}_i} A_{ij} {\bm y}_j^k,
\end{align}
\end{subequations}
where $y_{i,i}^k$ denotes the $i$th element of the $n$-dimensional vector ${\bm y}_i^k$ and $\gamma_k > 0$ is the $k$th stepsize. While \eqref{iter:prm} resembles the `consensus \& update' paradigm in classical DGD \cite{nedic2009distributed}, it utilizes a weighted local gradient via the rescaling factor $( y_{i,i}^k n )^{-1}$. From \eqref{iter:y_k}, the latter factor approximates $( {\pi}_i({\bm A}) \, n )^{-1}$ which normalizes the local gradient $\nabla f_i(\prm_i^k)$ prior to taking the weighted average by ${\bm A}$.
To highlight the dependence on ${\bm A}$, we will denote the recursion in \eqref{strategy:1} by 
\beq 
( \Prm^{k+1} , {\bm Y}^{k+1} ) = {\sf DiDGD}( \Prm^k, {\bm Y}^k ; {\bm A}, \gamma_k ),
\eeq 
where $\Prm^k := \big[ \prm_1^k , \ldots, \prm_n^k \big]^\top$ and ${\bm Y}^k := [ {\bm y}_1^k, \ldots, {\bm y}_n^k ]^\top$.

We state the convergence results of Di-DGD \eqref{strategy:1} for tackling \eqref{P1}. Our analysis follows from a non-trivial extension of \cite{Xi-row} to the non-convex setting with Di-DGD, and provides comparable rate to that in \cite{assran2019stochastic}. We consider the following set of assumptions:
\begin{assumption}
    \label{assu:L_Lip}
    For $i\in\{1,2,\dots,n\}$, there exists $L$ such that
    \begin{equation*}
        \| \nabla f_i(\prm_1) - \nabla f_i(\prm_2) \| \leq  L \| \prm_1 - \prm_2 \|, \quad \forall ~ \prm_1, \prm_2 \in \mathbb{R}^d.
    \end{equation*}
\end{assumption}
\begin{assumption}
\label{assu:bdd_grad}
   For $i\in\{1,2,\dots,n\}$, there exists $G$ such that 
   \begin{equation*}
       \| \nabla f_i(\prm) \| \leq G < \infty, \quad \forall ~ \prm \in \mathbb{R}^d.
   \end{equation*}
\end{assumption}
\begin{assumption}
    \label{assu:data_hetero}
    For $i\in\{1,2,\dots,n\}$, there exists $\varsigma$ such that
    \begin{equation*}
        \| \nabla F(\prm) - \nabla f_i(\prm) \| \leq \varsigma < \infty, \quad \forall ~ \prm \in \mathbb{R}^d.
    \end{equation*}
\end{assumption}
\noindent Assumption \ref{assu:L_Lip} corresponds to the standard smoothness condition and Assumption \ref{assu:bdd_grad} ensures the boundedness of gradients, while Assumption \ref{assu:data_hetero} quantifies the heterogeneity of local data held by different agents. These are standard assumptions in the literature of non-convex decentralized optimization, see \cite{chang2020distributed}.

We establish the following convergence result for Di-DGD \eqref{strategy:1} on non-convex problems:
\begin{theorem}\label{theorem:DiDGD}
Under Assumptions \ref{assu:strong_connect}-\ref{assu:data_hetero}. Define $C_{\pi 1} = \sum_{i=1}^n (1-\pi_i)^2 + (n-1) \pi_i^2 $, $C_{\pi 2} = \sum_{i=1}^n \frac{1}{\pi^{2}_i}$.
Let $\gamma_k \leq \frac{n\rho_{\bm A}}{2L\sqrt{C_{\pi1}(C_{\pi2}+C_0)}}$ for some constant $C_0$. Then, for any $T \geq 1$, the iterates generated by the  Di-DGD algorithm \eqref{strategy:1} satisfy
\begin{equation}
\begin{split}
& \textstyle \min_{k=0, \ldots, T-1} \frac{1}{n} \sum_{i=1}^n \| \nabla F(\prm_i^k) \|^2 \\
& = {\cal O} \left( \frac{ F( \Hprm^0 ) - F(\Hprm^T)  + \frac{\varsigma^2 C_{\pi 1} C_{\pi 2}}{n^2\rho_{\bm A}^2} \sum_{k=0}^{T-1} \gamma_k^3 }{ \sum_{k=0}^{T-1} \gamma_k } \right) ,
\end{split}
\end{equation}
where $\Hprm^k = \sum_{i=1}^n \pi_{i, {\bm A}} \prm_i^k$ is the weighted average of $\{ \prm_i^k \}_{i=1}^n$.
\end{theorem}
\noindent 
The proof can be found in the {\href{https://www1.se.cuhk.edu.hk/~htwai/pdf/icassp26-d3gd.pdf}{Online Appendix}.
Notice that $\frac{1}{n} \sum_{i=1}^n \| \nabla F(\prm_i^k) \|^2$ measures the average stationarity level evaluated on $F(\cdot)$ for each local variable at the $k$th iteration.
By setting $\gamma_k = \gamma = {\cal O}(1/T^{1/3})$, we obtain a rate of ${\cal O}(1/T^{2/3})$ for the convergence to stationary point of \eqref{P1}. Lastly, we remark that this rate is comparable to that of DGD over undirected graphs \cite{nedic2009distributed} and is faster than that of \cite{assran2019stochastic} due to the use of exact gradients.

\section{Dynamic Di-DGD Algorithm}
\subsection{Centralized Dynamic Algorithm}
An important observation from Theorem \ref{theorem:DiDGD} is that the convergence rate of Di-DGD depends on the spectral gap $\rho_{\bm A}$ and the heterogeneity constant $\varsigma$. A small $\rho_{\bm A}$ (e.g., due to a poorly designed ${\bm A}$) or a large $\varsigma$ (e.g., due to highly heterogeneous local data) will lead to slow convergence. This motivates us to develop a strategic communication framework to accelerate the convergence of Di-DGD \eqref{strategy:1} by dynamically adjusting the adjacency matrix ${\bm A} \in {\cal A}_{\cal G}$.
In particular, this section proposes the \emph{dynamical Di-DGD} ({\algname}) framework which is based upon Di-DGD, but will adapt the topology dynamically according to similarities between gradients of agents. 

Our primary idea is to identify a cost-to-go function for Di-DGD \eqref{strategy:1} and extract the component that can be optimized with respect to (w.r.t.) the weighted adjacency matrix ${\bm A}$. 
We focus on finding the strategy, i.e., a refined weighted adjacency matrix ${\bm A}^k$, at iteration $k$ such that the iterate $\Prm^k$ is fixed. At the first glance, it may be natural to consider the average objective value $\frac{1}{n} \sum_{i=1}^n F( \prm_i^k ) = \frac{1}{n^2} \sum_{i,j=1}^n f_j( \prm_i^k )$. However, the latter may not yield a monotonic sequence w.r.t.~$k$, which will make it challenging to analyze. 

As a remedy, we are inspired by the proof of Theorem \ref{theorem:DiDGD} (see  {\href{https://www1.se.cuhk.edu.hk/~htwai/pdf/icassp26-d3gd.pdf}{Online Appendix}) to concentrate on the following Lyapunov function as the cost-to-go function. Let $\Hprm^k = \sum_{i=1}^n \pi_{i, {\bm A}} \prm_i^k$, we set
\beq \textstyle
L_k := F( \Hprm^k ) + \frac{ 3 \gamma_k L^2 }{ n\, \rho } \sum_{i=1}^n \| \prm_i^k - \Hprm^k \|^2.
\eeq 
Define the notations $\grd {\bm F}^k := \big[ \grd f_1(\prm_1^k) , \ldots, \grd f_n(\prm_n^k) \big]^\top$, $\HPrm^k := {\bf 1} \bm{\pi}_{\bm A}^\top \Prm^k$, $\widetilde{\bm Y}^k := {\rm Diag}( {\bm Y}^k )$.
Using Assumptions \ref{assu:strong_connect}-\ref{assu:data_hetero}, it can be proven that (see {\href{https://www1.se.cuhk.edu.hk/~htwai/pdf/icassp26-d3gd.pdf}{Online Appendix})
\begin{align}
L_{k+1} & \leq L_k - \frac{\gamma_k}{4} \| \grd F( \Hprm^k ) \|^2 - \frac{3 \gamma_k L^2 (2-\rho)}{n \, \rho} \| \Prm^k - \HPrm^k \|_F^2 \notag \\
& \quad + {\cal O}( \gamma_k^3 ) + \frac{3 \gamma_k L^2}{n \, \rho} \| \Prm^{k+1} - \HPrm^{k+1} \|_F^2 . \label{eq:1stLya}
\end{align}
Our goal is not to prove the convergence of Di-DGD, but to \emph{construct a design function that shall guide the optimization of ${\bm A}$ for faster Di-DGD convergence}. 
As such, instead of attempting at the convergence of $L_k$, we focus on {designing ${\bm A}$ to be used at the $k$th iteration such that $L_{k+1}$ is minimized.} 
To this end, we notice that as $\Prm^k$ is fixed prior to iteration $k$, only the last term on the right hand side of \eqref{eq:1stLya} depends on ${\bm A}$. Expanding the latter yields:
\begin{align}
& \| \Prm^{k+1} - \HPrm^{k+1} \|_F^2 = \| ({\bm I} - {\bf 1} \bm{\pi}_{\bm A}^\top) ( {\bm A} \Prm^k - {\textstyle \frac{\gamma_k}{n}} (\widetilde{\bm Y}^k)^{-1} \nabla {\bm F}^k)  \|_F^2 \notag \\
& = \frac{\gamma_k^2}{n^2} \| (\tilde{\bm Y}^k)^{-1} ( {\bm I} - {\bf 1} \bm{\pi}_{\bm A}^\top ) \nabla {\bm F}^k \|_F^2 + \| ( {\bm A} - {\bf 1} \bm{\pi}_{\bm A}^\top  ) \Prm^k \|_F^2  \notag \\ 
& \quad - \frac{2\gamma_k}{n} \langle  ( {\bm A} - {\bf 1}_n \bm{\pi}_{\bm A}^\top ) \Prm^k ,   ( {\bm I} - {\bf 1} \bm{\pi}_{\bm A}^\top ) (\tilde{\bm Y}^k)^{-1} \nabla {\bm F}^k) \rangle.
\end{align}
Ignoring the first ${\cal O}( \gamma_k^2 )$ term and fixing $\bm{\pi}_{\bm A} = \bm{\pi}_{\bm A^k}$ leads to the design function for ${\bm A}$ at the $k$th iteration:
\beq 
\begin{split}
\label{eq:design function of Jk}
& J_k( {\bm A} ; \Prm^k ) := \| ( {\bm A} - {\bf 1} \bm{\pi}_{\bm A^k}^\top  ) \Prm^k \|_F^2 \\
& - \frac{2\gamma_k}{n} \langle  ( {\bm A} - {\bf 1}_n \bm{\pi}_{\bm A^k}^\top ) \Prm^k ,  ( {\bm I} - {\bf 1} \bm{\pi}_{\bm A^k}^\top ) (\tilde{\bm Y}^k)^{-1} \nabla {\bm F}^k) \rangle.
\end{split}
\eeq 
As anticipated, the design function depends on $\Prm^k$ and the similarities between the set of local gradients $\{ \grd f_i( \prm_i^k ) \}_{i=1}^n$. 

Minimizing the design function $J_k(\cdot)$ enables us to maximize the decrease in the cost-to-go function $L_{k+1}$, i.e., maximizing the convergence rate at the $k$th iteration. We are motivated to consider the following dynamical variant of Di-DGD:
\beq \label{eq:naivedecay}
\begin{split} 
& {\bm A}^{k+1} = \textstyle  \argmin_{ {\bm A} \in {\cal A}_{\cal G} } J_k( {\bm A} ; \Prm^k ), \\
& ( \Prm^{k+1}, {\bm Y}^{k+1} ) = {\sf DiDGD} ( \Prm^k, {\bm Y}^k ; {\bm A}^{k+1} , \gamma_k ).
\end{split}
\eeq 
However, \eqref{eq:naivedecay} may not work since (i) the optimized ${\bm A}^{k+1}$ may violate the irreducibility assumption in Assumption~\ref{assu:strong_connect} which can invalidate the cost-to-go function $L_k$, (ii) the cost-to-go function $L_k$ will be updated at every iteration as the Perron eigenvector is perturbed due to changes in ${\bm A}^k$, i.e., one has $\bm{\pi}_{\bm A^{k+1}} \neq \bm{\pi}_{\bm A^{k}}$, and (iii) minimizing $J_k(\cdot)$ requires \emph{global information} involving all local variables $\Prm^k, {\bm Y}^k, \grd {\bm F}^k$.

\subsection{Practical D$^3$GD Algorithms}
The proposed \emph{dynamical Di-DGD} ({\algname}) framework is based on the decentralized algorithm in \eqref{eq:naivedecay}. To address the issues mentioned in previous section, we rely on two ingredients to design ${\bm A}^k$ on-the-fly: {\sf (a)} a slow-timescale update for ${\bm A}^k$ to prevent violating  Assumption~\ref{assu:strong_connect} and causing perturbations to $\bm{\pi}_{\bm A^{k}}$, {\sf (b)} a coordinate descent algorithm for optimizing each row of ${\bm A}^k$, that corresponds to the communication strategy of an agent, in a decentralized manner. 

We first concentrate on the satisfaction of Assumption~\ref{assu:strong_connect} by adopting a conservative design for ${\bm A}^k$. We consider
\beq \textstyle \min_{ \overline{\bm A} \in {\cal A}_{\cal G} } \overline{J}_k( \overline{\bm A}; \Prm^k ) := J_k( (1-\delta) \overline{\bm A} + \delta {\bm A}^0 ; \Prm^k ),
\eeq 
with $\delta \in (0,1)$ being a pre-defined weight and ${\bm A}^0 \in {\cal A}_{\cal G}$ is an initial adjacency matrix satisfying Assumption~\ref{assu:strong_connect}.
Subsequently, the convex combination between an optimized $\overline{\bm A}^k$ and the initial ${\bm A}^0$, i.e., ${\bm A}^k = (1-\delta) \overline{\bm A}^k + \delta {\bm A}^0$ must satisfy Assumption~\ref{assu:strong_connect}. 

We then consider a gradient descent scheme for $\overline{\bm A}^k$ to develop a {\it decentralized} mechanism where each agent optimizes his/her own communication strategy. To this end, we notice from \eqref{strategy:1} that knowing the $i$th row of $\overline{\bm A}^k$, denoted as $\overline{\bm a}_i^k$, suffices for agent $i$ to perform the update(s). At iteration $k$, the {\algname} algorithm relies on updating the $i$th row of $\overline{\bm A}^k$ via projected gradient method, given by:
\beq \label{eq:pgd}
\overline{\bm a}_{i}^{k+1} = {\cal P}_{ {\cal A}_{ {\cal G}, i} } \left( \overline{\bm a}_{i}^{k} - \eta_k \grd_{i} \overline{J}_{k} ( \overline{\bm A}^k ; \Prm^k) \right),
\eeq 
where $\eta_k > 0$ is a step size parameter and ${\cal P}_{ {\cal A}_{ {\cal G}, i } } (\cdot)$ denotes the projection onto ${\cal A}_{ {\cal G}, i }$ that corresponds to the $i$th row of the matrix in ${\cal A}_{\cal G}$. We have ${\cal A}_{\cal G} = {\cal A}_{ {\cal G}, 1 } \times \cdots \times {\cal A}_{ {\cal G}, n }$. Note that we also consider coordinate descent update such that \eqref{eq:pgd} is executed for a set of rows $i \in {\cal I}_k$, while we set $\overline{\bm a}_{j}^{k+1} = \overline{\bm a}_{j}^{k}$, $j \notin {\cal I}_k$. 
By judiciously designing the parameters $\eta, \delta$ together with the Di-DGD step size $\gamma$, we ensure a smooth evolution for $\bm{\pi}_{\bm A^{k+1}} \approx \bm{\pi}_{\bm A^{k}}$.

The last issue is that \eqref{eq:pgd} requires {\it global information} not available at the individual agent. Observe that for any $j \in {\cal N}_i$,
\begin{align}
& \big[ \grd_{i} \overline{J}_{k} ( \overline{\bm A} ; \Prm^k) \big]_j  = 2 (\prm_j^k)^\top (\Prm^k)^\top ( (1-\delta) \overline{\bm a}_i + \delta {\bm a}_i^0 - \bm{\pi}_{\bm A^k} ) \notag \\
& \textstyle \quad - \frac{2\gamma_k (1-\delta) }{n} (\prm_j^k)^\top ( \grd {\bm F}^k )^\top ( \tilde{\bm Y}^k )^{-1} ( {\bm e}_i - \bm{\pi}_{\bm A^k} ) \label{eq:grd_J}
\end{align}
In particular, evaluating $(\Prm^k)^\top \bm{\pi}_{\bm A^k}$ and $( \grd {\bm F}^k )^\top ( \tilde{\bm Y}^k )^{-1} \bm{\pi}_{\bm A^k}$ may involve $\prm_{\ell}^k$, $\grd f_\ell( \prm_\ell^k )$ with $\ell \notin {\cal N}_i$.
To this end, our idea is to approximate the latter by constructing two decentralized trackers similar to dynamic consensus \cite{zhu2010discrete}. We have
\begin{align}
{\bm z}_i^{k+1} & \textstyle = \sum_{j=1}^n A_{ij}^k {\bm z}_j^k + \prm_i^{k+1} - \prm_i^k \label{eq:tracker_z} \\
{\bm q}_i^{k+1} & \textstyle = \sum_{j=1}^n A_{ij}^k {\bm q}_j^k + \grd f_i( \prm_i^{k+1} ) - \grd f_i( \prm_i^{k}) \label{eq:tracker_q}
\end{align}
that track $(\Prm^{k+1})^\top \bm{\pi}_{\bm A^k}$ and $( \grd {\bm F}^{k+1} )^\top \bm{\pi}_{\bm A^k}$, respectively. 
Finally, we summarize the fully decentralized {\algname} in Algorithm~\ref{alg:d3gd}.

\begin{algorithm}[t]
\algrenewcommand\algorithmicrequire{\textbf{Input:}}
\algrenewcommand\algorithmicensure{\textbf{Output:}}
\algrenewcommand{\algorithmicindent}{1em}
\caption{Decentralized Dynamic Di-DGD ({\algname}) Algorithm}
\label{alg:d3gd}
\begin{algorithmic}[1]
\Require Initial parameters $\{ \prm_i^0 \}_{i=1}^n$, initial adjacency matrix ${\bm A}^0$, step-sizes $\{\gamma_k, \eta_k\}_{k \geq 0}$, max no.~of iterations $T$.
\State Set ${\bm z}_i^0 = \prm_i^0$, ${\bm q}_i^0 = \grd f_i( \prm_i^0 )$, $i=1,...,n$, and $\overline{\bm A}^0 = {\bm A}^0$.
\For{$k=0,1,2,\ldots,T-1$}
\State \texttt{/* Di-DGD update */}
\State Set ${\bm A}^k = (1-\delta) \overline{\bm A}^k + \delta {\bm A}^0$.
\State $( \Prm^{k+1} , {\bm Y}^{k+1} ) = {\sf DiDGD} ( \Prm^k , {\bm Y}^k ; {\bm A}^k , \gamma_k ) $
\State \texttt{/* Adjacency Matrix update */}
\For{$i=1,\ldots,n$ (in parallel)}
\State \vspace{-.6cm}
\begin{align}
& \hspace{.45cm} \textstyle \forall~j \in {\cal N}_i,~\big[ {\bm g}_i^k \big]_j  
 = \textstyle 2 (\prm_j^k)^\top ( \frac{\gamma_k (1-\delta) }{n Y_{ii}^k} {\bm q}_i^k - {\bm z}_i^k ) \, + \notag \\ 
 & \hspace{.45cm}  \textstyle 2 (\prm_j^k)^\top \big[ (\Prm^k)^\top ( (1-\delta) \overline{\bm a}_i^k + \delta {\bm a}_i^0 ) - \frac{\gamma_k (1-\delta) }{n Y_{ii}^k} \grd f_i( \prm_i^k ) \big] \notag 
\end{align}\vspace{-.4cm}
\State Set $\overline{\bm a}_i^{k+1} = {\cal P}_{ {\cal A}_{ {\cal G}, i} } ( \overline{\bm a}_i^{k} - \eta_k {\bm g}_i^k )$.
\State Update the trackers ${\bm z}_i^k, {\bm q}_i^k$ using \eqref{eq:tracker_z}, \eqref{eq:tracker_q}.
\EndFor
\EndFor
\Ensure Approximate solution $\{ \prm_i^T \}_{i=1}^n$, adjacency matrix ${\bm A}^T$.
\end{algorithmic}
\end{algorithm}

We conclude with a discussion on the convergence property of {\algname}. Notice that Algorithm~\ref{alg:d3gd} operates on two optimization problems simultaneously: one on Di-DGD \eqref{strategy:1} that aims at tackling \eqref{P1} and one on minimizing the design function $\overline{J}(\cdot)$. Although the two problems are intertwined with each other, we observe that when $\overline{\bm A}$ is unchanged, the Di-DGD algorithm is guaranteed to converge to a stationary solution (cf.~Theorem~\ref{theorem:DiDGD}). Meanwhile, the design function $\overline{J}(\cdot)$ is convex w.r.t.~the adjacency matrix. We thus anticipate {\algname} to converge under a carefully adjusted stepsize pair $(\gamma_k, \eta_k)$.

\section{Numerical Experiments}
This section presents numerical results to validate the efficacy of {\algname} in different settings. Throughout, we focus on the setting when the learning problem \eqref{P1} is a robust $K$-class classifier training problem with synthetic data. Particularly, we take $K=10$ classes and for each $i$, the set of parameters is $\prm = ( \prm^{(k)} )_{k=1}^K$, $\prm^{(k)} \in \RR^d$. Consider the sigmoid loss function
\begin{equation} \notag
    f_i(\prm) = \frac{1}{M} \sum_{m=1}^{M} \sum_{k=1}^{K} \frac{1}{1 + \exp(-y_{mk} (\mathbf{x}_m^\top \prm^{(k)} ))} + \frac{\lambda}{2} \sum_{k=1}^{K} \| \prm^{(k)} \|_2^2.
\end{equation}
We set $\lambda = 10^{-4}$ as a regularization parameter and notice that $f_i(\cdot)$ is smooth but possibly non-convex. Here, agent $i$ has access to $M$ samples of data tuples $( {\bf x}_m, (y_{mk})_{k=1}^{K} )_{m=1}^{M}$ such that ${\bf x}_m \in \RR^d$ is the $m$th feature vector and $y_{mk} = 1$, if $k$ is the respective class label for ${\bf x}_m$, otherwise, $y_{mk} = 0$. 

To generate heterogeneous data across agents, for each $i$, we first sample a class probability vector $\mathbf{p}_i = (p_{i,1}, \dots, p_{i,K})$ with the p.d.f.~$f(\mathbf{p}_i; \alpha) = \frac{\Gamma(K\alpha)}{\Gamma(\alpha)^K} \prod_{k=1}^{K} p_{i,k}^{\alpha - 1}$, where $\sum_{k=1}^K p_{i,k} = 1$ and $\Gamma(\cdot)$ is the Gamma function. A small $\alpha$ (i.e., $ \alpha \approx 0$) leads to sparse probability vectors and thus high degree of data heterogeneity, while a large $\alpha$ (i.e., $\alpha \gg 1$) forces $\mathbf{p}_i$ to be nearly uniform, leading to low degree of heterogeneity. The exact number of samples for each class is obtained by drawing a vector of counts $\mathbf{m}_i$, s.t. $P(\mathbf{m}_i) = \frac{m!}{m_{i,1}! \dots m_{i,K}!} p_{i,1}^{m_{i,1}} \dots p_{i,K}^{m_{i,K}}$, where $\sum_{k=1}^K m_{i,k} = M$. We set $M=100$ to be the total number of samples at each agent. Finally, we generate a unique mean vector $\boldsymbol{\mu}_k \in \mathbb{R}^d$ for each class $k$ by $\boldsymbol{\mu}_k \sim \mathcal{N}(\mathbf{0}, \mathbf{I}_d)$, where $d=10$. For each class $k$, the $m_{i,k}$ feature vectors $\mathbf{x}$ for agent $i$ are then sampled from a normal distribution centered at the corresponding class mean: $\mathbf{x} \sim \mathcal{N}(\boldsymbol{\mu}_k, \mathbf{I}_d)$.

In our experiments, {\algname} is implemented with $\delta = 0.2$ and full participation in updating $\overline{\bm a}_i^k$, i.e., ${\cal I}_k = {\cal V}$. We consider two variants of {\algname}: (i) the plain `{\algname}' utilizes \eqref{eq:pgd} to update $\overline{\bm A}^k$, (ii) `{\it decentralized {\algname}}' refers to the fully decentralized scheme in Algorithm~\ref{alg:d3gd}. We also compared the Di-DGD algorithm with the standard Metropolish-Hasting weight \cite{metropolis_weights}.


\vspace{.2cm}
\noindent {\bf Convergence of {\algname} on Synthetic Data.}
Our first experiment compares {\algname} to Di-DGD for tackling an instance of \eqref{P1} with $n = 20$ agents, connected via a directed Erdos-Renyi (ER) graph with connectivity of $p=0.6$. The data heterogeneity level is set at $\alpha = 0.1$. We compare the stationarity measure $\frac{1}{n}\sum_{i=1}^n \| \grd F(\prm_i^t) \|^2$ and total disagreement $\frac{1}{n^2} \sum_{i,j} \| \prm_i^k - \prm_j^k \|^2$ against $k$. The results are presented in Fig.~\ref{fig:strategy comparison} where we have set the stepsize parameters as $(\gamma_k,\eta_k) = (0.1,1)$ for {\algname}.
From the figure, we notice that both variants of {\algname} achieve speed-ups of 30-40\% over Di-DGD. Moreover, the fully decentralized {\algname} achieves similar performance as its global information counterpart, and maintains a similar level of speedup over Di-DGD.
The experiment supports our conjecture that the convergence of (directed graph) distributed learning can be accelerated by dynamically refining the graph topology. 

\begin{figure}[t]

\begin{minipage}[b]{.48\linewidth}
  \centering
  \centerline{\includegraphics[width=4.0cm]{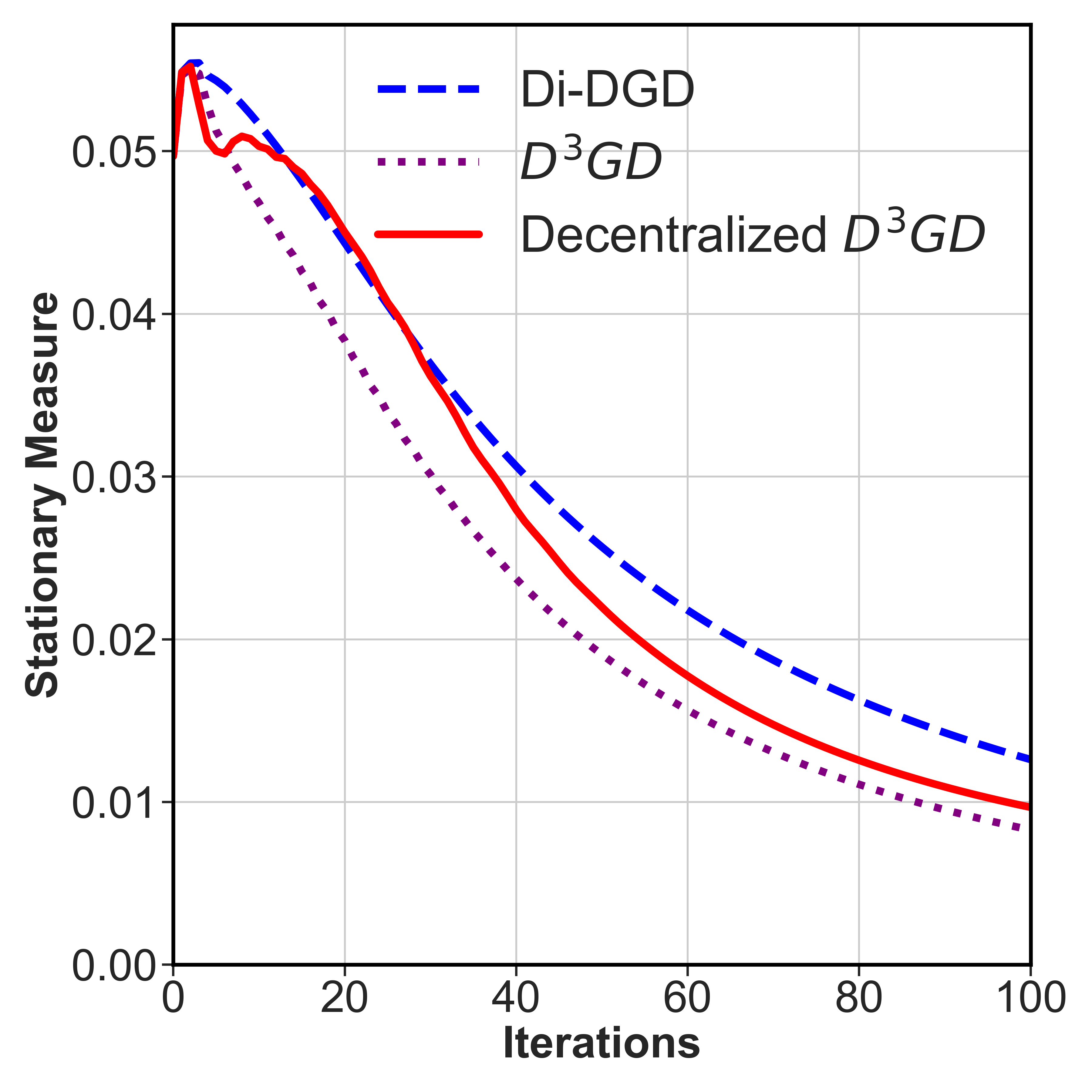}}
  
\end{minipage}
\hfill
\begin{minipage}[b]{0.48\linewidth}
  \centering
  \centerline{\includegraphics[width=4.0cm]{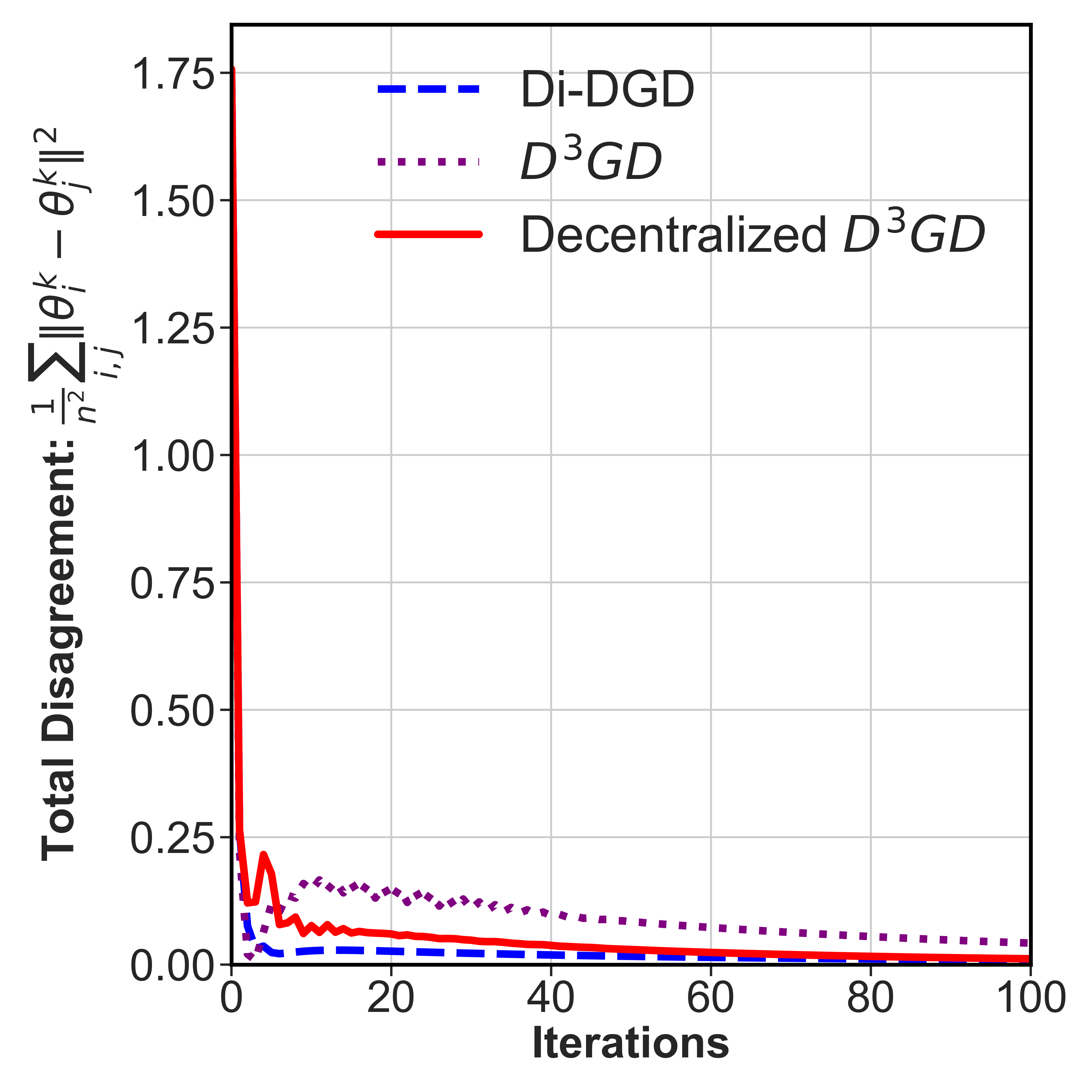}}
  
\end{minipage}
\caption{Convergence of {\algname} and Di-DGD on sigmoid loss minimization. (Left) Stationarity measure $\frac{1}{n} \sum_{i=1}^n \| \grd F( \prm_i^k ) \|^2$. (Right) Total disagreement $\frac{1}{n^2}\sum_{i,j} \| \prm_i^t - \prm_j^t \|^2$.}
\label{fig:strategy comparison}
\end{figure}

\begin{figure}[t] 
    \centering 

    \begin{subfigure}[b]{0.48\linewidth} 
        \centering
       
        \includegraphics[width=\textwidth]{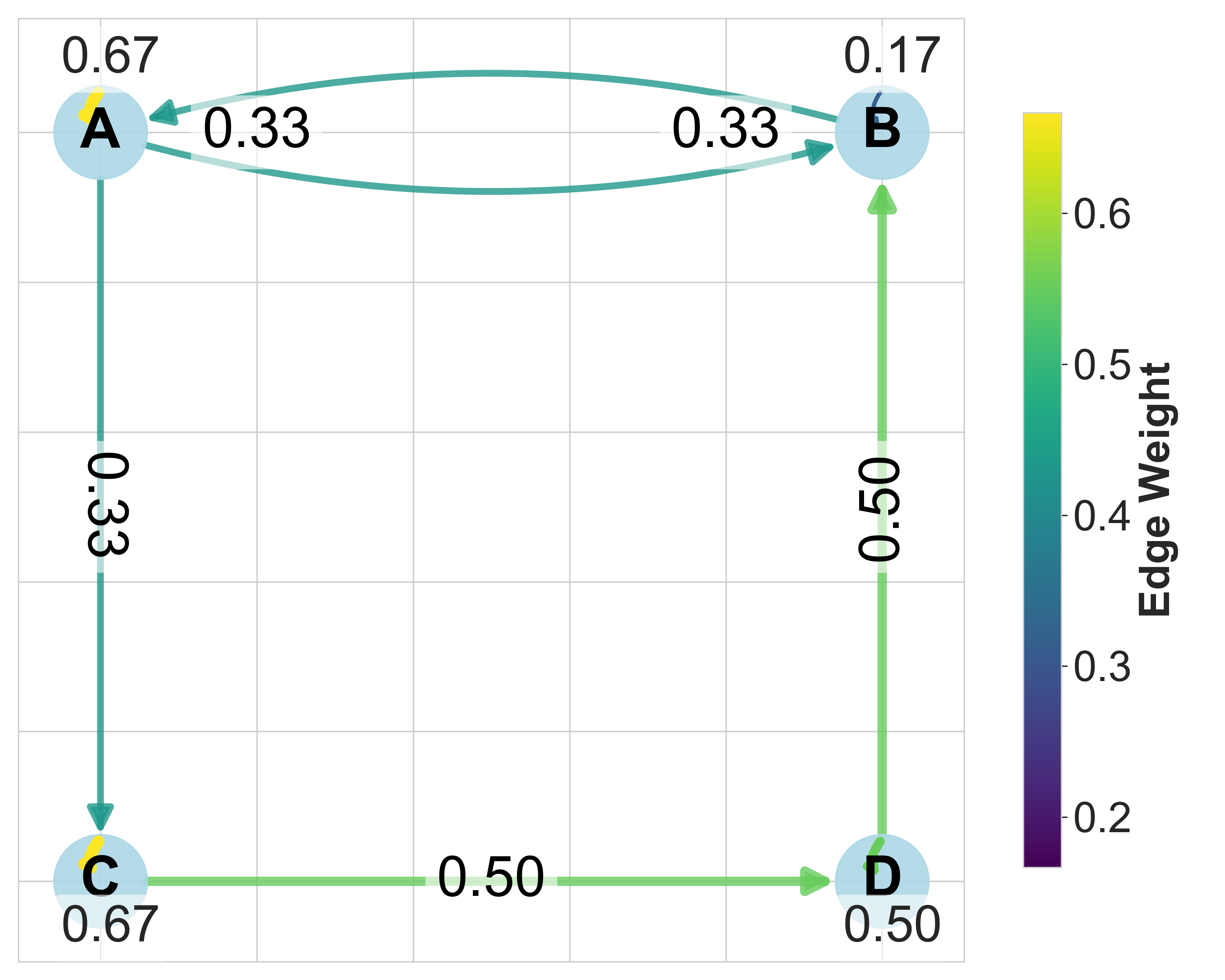}
        \caption{Edge weights of a directed ring topology with Metropolis weights.}
        \label{fig:initial_graph} 
    \end{subfigure}
    \hfill 
    \begin{subfigure}[b]{0.48\linewidth} 
        \centering
        \includegraphics[width=\textwidth]{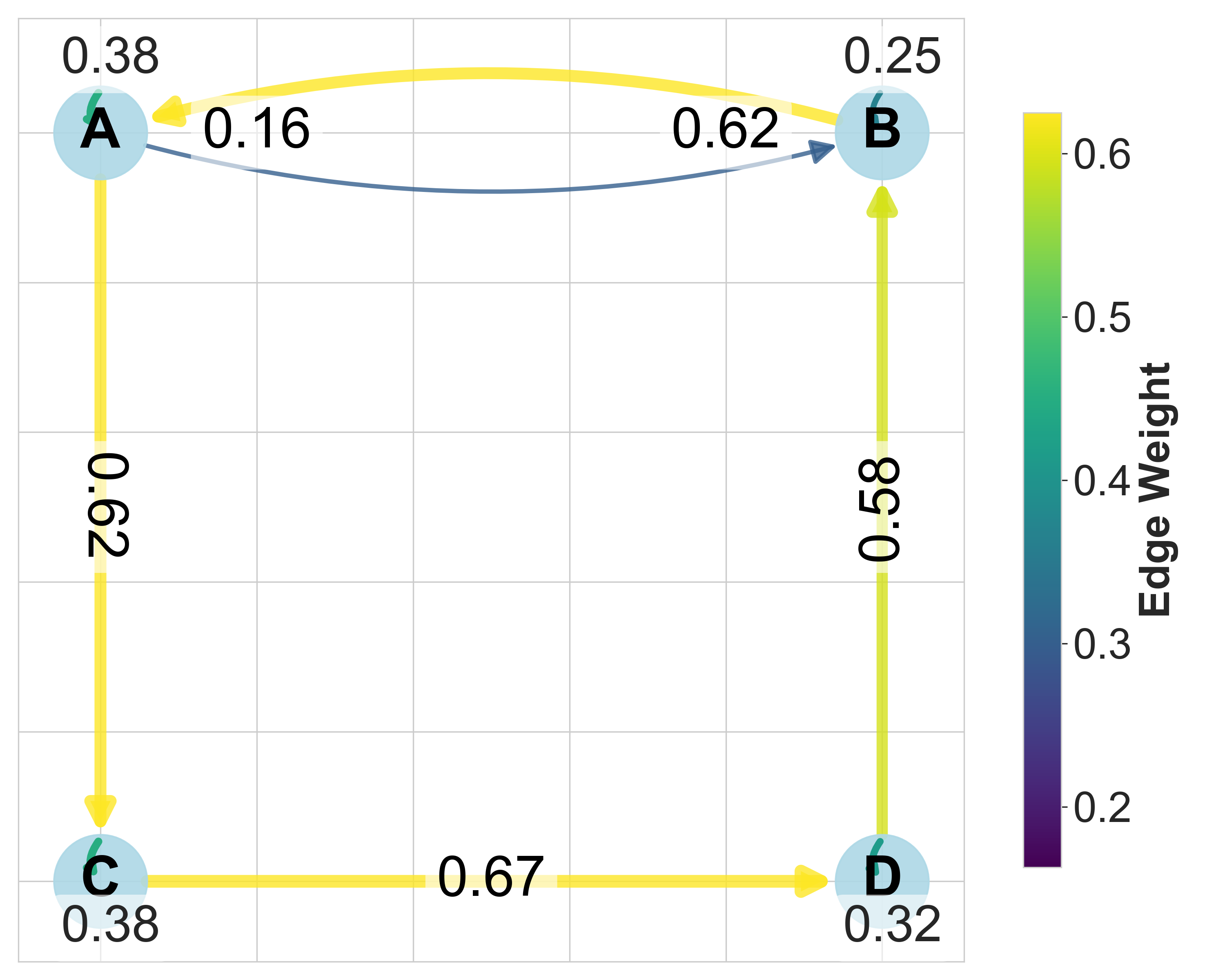}
        \caption{Edge weights obtained by {\algname} after 1000 iterations.}
        \label{fig:final_graph} 
    \end{subfigure}
    \caption{Visualizations of the edge weights.}
    \label{fig:graphs_comparison} 
\end{figure}

\vspace{.2cm}
\noindent{\bf Optimized Graph Topology.}
The second set of experiments delves into investigating the graph weights learnt by {\algname}. For illustration purpose, we concentrate on a simple directed graph of $n=4$ agents with the topology ${\cal G}$ in Fig.~\ref{fig:initial_graph} together with the initial edge weights. We designate agent {\sf A} as the outlier with heterogeneous data, while the data distributions of the other three agents are nearly homogeneous. This is achieved by setting the parameter in synthetic data generation $\alpha=0.1$  for agent {\sf A} and $\alpha=100$ for agents {\sf B, C, D}.

Fig.~\ref{fig:final_graph} shows the edge weights learnt by {\algname} after $T=1000$ iterations. 
Notice that the optimized edge weight has assigned increased weights for the edges of ${\sf A} \to {\sf C}$, ${\sf C} \to {\sf D}$, and ${\sf B} \to {\sf A}$. The result is reasonable since in order to achieve consensus, information from {\sf A} that is heterogeneous must be propagated to other agents. 
Obviously, {\sf B} and {\sf C} can receive information from {\sf A} directly, yet {\sf D} can only receive information from {\sf A} indirectly via {\sf C}. Therefore, it is anticipated that the information flow from {\sf A} to {\sf B} can be partially replaced by the information flow provided by {\sf D}. 
Subsequently, information transmission along the ${\sf A} \to {\sf C}$ and ${\sf C} \to {\sf D}$ paths becomes critical. 
This simple experiment allows us to observe the relationship between the underlying data distribution and the graph topology optimized by the proposed {\algname}. That is, the higher the data heterogeneity between a pair of agents, the higher their connecting edge weight tends to be. It confirms our conjecture that an optimized topology shall be data dependent

\vspace{-.1cm}
\section{Conclusion}\vspace{-.1cm}
We have proposed {\algname}, a data-dependent algorithm that dynamically refines communication edge weights in decentralized learning on directed graphs. We design {\algname} from a Lyapunov analysis to guide the optimization of these weights and develop a fully decentralized implementation using locally available information. The numerical experiments show that {\algname} achieves speed-ups over Di-DGD. We finally reveal the potential relationship between the underlying data distribution and the graph optimized by {\algname}.

\newpage
\bibliographystyle{IEEEtran}
\bibliography{strings,refs}

\newpage

\onecolumn
\appendix

\section{Preliminary Lemmas}
\newcommand{\C}{{\mathbb{C}}}

\begin{lemma} \cite[Lemma 1]{Xi-row}
    \label{lemma 1}
    Under Assumption~\ref{assu:strong_connect} and consider the recursion for ${\bm Y}^k$ in \eqref{iter:y_k}. There exists a constant $\lambda \in [0,1)$, $\C < \infty$ such that for any $k \geq 0$,
    \begin{equation*}
        \| {\bm Y}^k - {\bf 1} \bm{\pi}_{\bm A}^\top\|_2 \leq \C \lambda^k.
    \end{equation*}
\end{lemma}

\noindent
\begin{lemma}
\label{lemma 2}
Under Assumption~\ref{assu:strong_connect} and consider the recursion for ${\bm Y}^k$ in \eqref{iter:y_k}. There exists $C_0 < \infty$ such that we have the following upper bound
\begin{equation*}
    \| (\tilde{\bm Y}^k)^{-1} \|_F^2  \leq \sum_{i=1}^n \frac{1}{\pi^{2}_{i, {\bm A}} } +C_0 \lambda^k ,
\end{equation*}
where the $\lambda \in [0,1)$ was the constant defined in Lemma \ref{lemma 1}.
\end{lemma}

\begin{proof}
From Lemma \ref{lemma 1}, the diagonal entry $Y_{ii}^k$ converges geometrically to $\pi_{i, {\bm A}}$, i.e., $|Y_{ii}^k - \pi_{i, {\bm A}} | \le \C \lambda^k$. For any finite $k$, we define $Y_{min} = \inf_{k \geq 0,i} Y_{ii}^k $, and obviously $Y_{ii}^k \leq 1$. And we define $\pi_{min} = \inf_{i} \pi_i$; there is also $\pi_i < 1$.
\begin{equation}
    \left| \frac{1}{(Y_{ii}^k)^2} - \frac{1}{\pi_i^2} \right| = \left| \frac{(Y_{ii}^k+ \pi_i)(Y_{ii}^k- \pi_i)}{(Y_{ii}^k)^2(\pi_i)^2} \right| \leq \frac{|(Y_{ii}^k+ \pi_i)||(Y_{ii}^k- \pi_i)|}{(Y_{ii}^k)^2(\pi_i)^2} \leq \frac{2 \C \lambda^k}{Y_{min}^2\pi_{min}^2}.
\end{equation}
Thus
\begin{equation}
    \frac{1}{(Y_{ii}^k)^2} \leq \frac{1}{\pi_i^2} + \frac{2 \C \lambda^k}{Y_{min}^2\pi_{min}^2} ,
\end{equation}
and  we can bound $\| (\tilde{\bm Y}^k)^{-1} \|_F^2$ by
\begin{equation}
    \| (\tilde{\bm Y}^k)^{-1} \|_F^2 = \sum_{i=1}^n  \frac{1}{(Y_{ii}^k)^2} \leq \sum_{i=1}^n \frac{1}{\pi_i^2} + C_0 \lambda^k, 
\end{equation}
where $C_0 = \sum_{i=1}^n  \frac{2 \C \lambda^k}{Y_{min}^2\pi_{min}^2}$.



\end{proof}

\section{Proof of Theorem~2.5}
Under Assumption \ref{assu:L_Lip} and using the update formula $\Hprm^{k+1} = \Hprm^k - \gamma_k \sum_{i=1}^n \frac{\pi_i}{nY_{ii}^k}\grd f_i(\prm_i^k)$, we have:
\begin{align}
\label{appendix eq : taylor 1}
F(\Hprm^{k+1}) &\le F(\Hprm^k) + \langle \grd F(\Hprm^k), \hat{\prm}^{k+1} - \Hprm^k \rangle + \frac{L}{2}||\Hprm^{k+1} - \hat{\prm}^k||^2  \notag \\
&\le F(\Hprm^k) -\gamma_k \left\langle \grd F(\hat{\prm}^k), \sum_{i=1}^n \frac{\pi_i}{nY_{ii}^k}\grd f_i(\prm_i^k) \right\rangle + \frac{L\gamma_k^2}{2} \left||\sum_{i=1}^n \frac{\pi_i}{nY_{ii}^k}\grd f_i(\prm_i^k) \right||^2
\end{align}
Notice that as we have defined $\Hprm^k := \sum_{i=1}^n \pi_i \prm_i^k$, we observe
\begin{align*}
    \left \langle \nabla F(\Hprm^k) , \sum_{i=1}^n \frac{\pi_i}{Y_{ii}^k \, n} \cdot \nabla f_i(\prm_i^k)  \right \rangle &= \left \langle \nabla F(\Hprm^k) , \sum_{i=1}^n \frac{\pi_i}{Y_{ii}^k \, n} \cdot \nabla f_i(\prm_i^k) - \nabla F( \Hprm^k ) + \nabla F( \Hprm^k ) \right \rangle \\
    & \geq \| \nabla F(\Hprm^k) \|^2 - \frac{1}{2} \| \nabla F(\Hprm^k) \|^2 - \frac{1}{2} \|  \sum_{i=1}^n \frac{\pi_i}{Y_{ii}^k \, n} \, \nabla f_i(\prm_i^k) - \nabla F(\Hprm^k) \|^2,
\end{align*}
and
\begin{align*}
    \left\| \sum_{i=1}^n \frac{\pi_i}{Y_{ii}^k \, n} \cdot \nabla f_i(\prm_i^k) \right\|^2 & \leq
    2 \| \nabla F(\Hprm^k) \|^2 + 2 \left\| \sum_{i=1}^n \frac{\pi_i}{Y_{ii}^k \, n} \cdot \nabla f_i(\prm_i^k) - \nabla F(\Hprm^k) \right\|^2 .
\end{align*}
The second term in the above can be controlled by the weighted consensus error as
\begin{align*}
    \| \sum_{i=1}^n \frac{\pi_i}{Y_{ii}^k \, n} \cdot \nabla f_i(\prm_i^k) - \nabla F(\Hprm^k) \|^2 &= \| \sum_{i=1}^n (\frac{\pi_i}{Y_{ii}^k \, n }+\frac{1}{n} - \frac{1}{n}) \cdot \nabla f_i(\prm_i^k) - \nabla F(\Hprm^k) \|^2 \\
    & \leq 2n \sum_{i=1}^n \| (\frac{\pi_i}{Y_{ii}^k \, n } - \frac{1}{n}) \cdot \nabla f_i(\prm_i^k) \|^2 + 2n \cdot \frac{1}{n^2} \sum_{i=1}^n \| \nabla f_i(\prm_i^k) - \nabla f_i(\hat{\prm^k})  \|^2 \\
    & \leq \frac{2}{n}  \sum_{i=1}^n \left(\frac{ \C \lambda^k}{Y_{ii}^k }\right)^2 \|\nabla f_i(\prm_i^k) \|^2 + \frac{ 2L^2 }{n} \sum_{i=1}^n \| \prm_i^k - \Hprm^k \|^2
\end{align*}
the last inequality follows from Lemma \ref{lemma 1}.
Substituting back into \eqref{appendix eq : taylor 1} gives 
  \begin{align}
        \label{eq:recursive F}
         F(\Hprm^{k+1}) - F(\Hprm^k) 
         & \leq ( - \frac{1}{2}\gamma_k + \gamma_k^2 L) \| \nabla F(\Hprm^k) \|^2 + (\gamma_k^2 L + \frac{\gamma_k}{2}) (\frac{2}{n}) \sum_{i=1}^n \left(\frac{C\lambda^k}{Y_{ii}^k } \right)^2 \|\nabla f_i(\prm_i^k) \|^2   \notag \\ 
         & \quad + (\gamma_k^2 L + \frac{\gamma_k}{2}) \frac{2L^2}{n} \sum_{i=1}^n  \| \prm_i^k - \Hprm^k \|^2 .
    \end{align}
By setting $\gamma_k \leq 1/(4L)$, we obtain
\beq \label{eq:recur_F_new}
F(\Hprm^{k+1}) - F(\Hprm^k) \leq - \frac{\gamma_k}{4} \| \grd F( \Hprm^k ) \|^2 + \frac{ 3 \gamma_k }{2n} \sum_{i=1}^n \left(\frac{C\lambda^k}{Y_{ii}^k } \right)^2 \|\nabla f_i(\prm_i^k) \|^2 + \frac{ 3 \gamma_k L^2}{2} {\textstyle \frac{1}{n}} \| \Prm^k - \HPrm^k \|_F^2.
\eeq 

Our next endeavor is to control the last term of consensus error. 
To do so, we aim to derive a recursion for $\frac{1}{n} \| \Prm^{k} - \HPrm^{k} \|^2_F$. Observe that
\begin{align*}
    \frac{1}{n} \| \Prm^{k+1} - \HPrm^{k+1} \|^2_F &= \frac{1}{n} \| ({\bm I} - {\bf 1} \bm{\pi}^\top ) \Prm^{k+1} \|_F^2 \\
    & = \frac{1}{n} \| ({\bm I} - {\bf 1} \bm{\pi}^\top ) ( {\bm A} \Prm^k - \gamma_k (n \tilde{\bm Y}^k)^{-1} \nabla {\bm F}^k)  \|_F^2 \\
    &= \frac{1}{n} \| ({\bm A} - {\bf 1} \bm{\pi}^\top ) \Prm^k - \frac{\gamma_k}{n}   ({\bm I} - {\bf 1} \bm{\pi}^\top ) (\tilde{\bm Y}^k)^{-1} \nabla {\bm F}^k \|_F^2 \\
    & \leq \frac{1+\alpha}{n} \| ({\bm A} - {\bf 1}_n \bm{\pi}^\top ) \Prm^k \|_F^2 + \frac{\gamma_k^2}{n^3} (1+ \frac{1}{\alpha}) \| ({\bm I} - {\bf 1} \bm{\pi}^\top ) (\tilde{\bm Y}^k)^{-1} \nabla {\bm F}^k  \|_F^2, 
\end{align*}
where the second inequality follows from that $\bm{\pi}$ is a left eigenvector of ${\bm A}$ such that $\bm{\pi}^\top {\bm A} = \bm{\pi}^\top$, and the last inequality follows from the Young's inequality, and $\alpha >0$ is a positive number to be fixed later.
As ${\bm A}$ is row-stochastic and $( {\bm A} - {\bf 1} \bm{\pi}^\top ) {\bf 1} = 0 $, we have
\begin{align*}
    \| ( {\bm A} - {\bf 1} \bm{\pi}^\top ) \Prm^k \|_F^2 &=   \| ( {\bm A} - {\bf 1} \bm{\pi}^\top ) ( {\bm I} - {\bf 1} \bm{\pi}^\top) \Prm^k  \|_F^2 \\
    & \leq \| ( {\bm A} - {\bf 1} \bm{\pi}^\top ) \|_F^2  \cdot \| ( {\bm I} - {\bf 1} \bm{\pi}^\top) \Prm^k \|_F^2 \\
    & = (1-\rho_{\bm A})^2 \| \Prm^k - \HPrm^k \|_F^2
\end{align*}
The last inequality is due to \eqref{eq:spectral}. For the second term, we notice that
\begin{align*}
    \| ({\bm I} - {\bf 1} \bm{\pi}^\top ) (\tilde{\bm Y}^k)^{-1} \nabla {\bm F}^k  \|_F^2 & \leq 2 \| ({\bm I} - {\bf 1} \bm{\pi}^\top )(\tilde{\bm Y}^k)^{-1} \left(\nabla {\bm F}^k - {\bf 1} \nabla F(\Hprm^k )^\top \right)  \|_F^2 + 2 \| ({\bm I} - {\bf 1} \bm{\pi}^\top )(\tilde{\bm Y}^k)^{-1} {\bf 1} \nabla F(\Hprm^k )^\top \|_F^2 
\end{align*}
We observe that 
\[
\| ({\bm I} - {\bf 1} \bm{\pi}^\top )(\tilde{\bm Y}^k)^{-1} {\bf 1} \nabla F(\Hprm^k )^\top \|_F^2 = \| ({\bm I} - {\bf 1} \bm{\pi}^\top )(\tilde{\bm Y}^k)^{-1} {\bf 1}\|^2 \| \nabla F(\Hprm^k) \|^2 \leq \left( \sum_{i=1}^n \frac{1}{\pi^{2}_{i, {\bm A}} } +C_0 \lambda^k \right) \| \nabla F(\Hprm^k) \|^2,
\]
where we have applied Lemma~\ref{lemma 2} in the last inequality. 
Moreover, using Assumptions \ref{assu:L_Lip} and \ref{assu:data_hetero},
\begin{align*}
\| ({\bm I} - {\bf 1} \bm{\pi}^\top )(\tilde{\bm Y}^k)^{-1} \left(\nabla {\bm F}^k - {\bf 1} \nabla F(\Hprm^k )^\top \right)  \|_F^2 & \leq C_{\pi 1} \| (\tilde{\bm Y}^k)^{-1} \|_F^2 \cdot \sum_{i=1}^n \left\| \nabla f_i(\prm_i^k) - \frac{1}{n} \sum_{j=1}^n \nabla f_j  (\Hprm^k) \right\|_F^2 \\
& \leq 2 C_{\pi 1} (C_{\pi 2} + C_0 \lambda^k ) \cdot(n\varsigma^2+L^2\|\Prm^k - \HPrm^k \|_F^2)
\end{align*}
where we have defined the constant $C_{\pi 1} = \sum_{i=1}^n (1-\pi_i)^2 + (n-1) \pi_i^2 $, $C_{\pi 2} = \sum_{i=1}^n \frac{1}{\pi^{2}_i}$. We therefore conclude that 
\begin{align*}
\| ({\bm I} - {\bf 1} \bm{\pi}^\top ) (\tilde{\bm Y}^k)^{-1} \nabla {\bm F}^k  \|_F^2 & \leq 4 C_{\pi 1} (C_{\pi 2} + C_0 \lambda^k ) \cdot(n\varsigma^2+L^2\|\Prm^k - \HPrm^k \|_F^2) + 2 \left( C_{\pi2} +C_0 \lambda^k \right) \| \nabla F(\Hprm^k) \|^2.
\end{align*}
We now set $\alpha = \frac{\rho_{\bm A}}{1-\rho_{\bm A}}$ and consider the case with $\gamma_k \leq \frac{n\rho_{\bm A}}{4L\sqrt{C_{\pi1}(C_{\pi2}+C_0)}}$, we arrive at the following recursion
\begin{equation}
\label{appendix eq:consensus recursion}
{\textstyle \frac{1}{n}} \| \Prm^{k+1} - \hat{\Prm}^{k+1} \|^2_F \leq (1-\frac{\rho_{\bm A}}{2}) {\textstyle \frac{1}{n}} \| \Prm^{k} - \hat{\Prm}^{k} \|^2_F + \left( \frac{4(C_{\pi 1}(C_{\pi 2}+C_0\lambda^k))}{n^2\rho_{\bm A}} \right)\gamma_k^2\varsigma^2 + \frac{2 \gamma_k^2 }{n^3 \rho_{\bm A}} \left( C_{\pi2} +C_0 \lambda^k \right) \| \nabla F(\Hprm^k) \|^2
\end{equation}


To bound the coupled evolution of $F(\Hprm^k)$ and $\| \Prm^t - \HPrm^t \|_F^2$, we consider the following Lyapunov function:
\begin{equation}
    \label{appendix eq: Lyapunov}
    L_k = F(\Hprm^k) + \frac{10\gamma_kL^2}{3n\rho_{\bm A}} \| \Prm^k - \HPrm^k \|^2_F .
\end{equation}
Using \eqref{eq:recur_F_new} and \eqref{appendix eq:consensus recursion}, the recursive equation of $L_k$ can be formulated as 
\begin{align*}
    L_{k+1} - L_k &= F(\Hprm^{k+1}) - F(\Hprm^k) + \frac{10\gamma_kL^2}{3n\rho_{\bm A}} \left( \| \Prm^{k+1} - \HPrm^{k+1} \|^2_F -\| \Prm^k - \HPrm^k \|^2_F\right) \\
    & \leq \left( \frac{-\gamma_k}{4} + \frac{20(C_{\pi2} + C_0\lambda^k)\gamma_k^3}{3n^3\rho^2_{\bm A}} \right) \| \nabla F(\Hprm^k) \|^2 - \frac{L^2}{6n}  \| \Prm^k - \HPrm^k \|^2_F + \left( \frac{40L^2(C_{\pi 1}(C_{\pi 2}+C_0\lambda^k))}{3n^2\rho^2_{\bm A}} \right)\gamma_k^3\varsigma^2 + \mathcal{O}(\gamma_k\lambda^{2k})
\end{align*}
Setting $\gamma_k^2 \leq n^3\rho^2_{\bm A}/ 80(C_{\pi 2}+\frac{C_0}{1-\lambda})$ and rearranging terms yield the following bound:
\begin{equation}
\label{eq: recursive of hat F and consensus}
    \gamma_k \left( \| \nabla F(\Hprm^k) \|^2 + \frac{L^2}{n} \| \Prm^k - \HPrm^k \|^2_F \right) \leq 6(L_k - L_{k+1}) + \left(\frac{80L^2(C_{\pi 1}(C_{\pi 2}+C_0\lambda^k))}{n^2\rho^2_{\bm A}} \right)\varsigma^2 \gamma_k^3 + \mathcal{O}(\gamma_k\lambda^{2k})
\end{equation}

Finally, we observe that for any $i = 1,\ldots,n$, we have
\begin{align}
   \| \nabla F(\prm_i^k) \|^2 & = \|\left( \grd F(\prm_i^k) - \grd F(\Hprm^k) \right) + \grd F(\Hprm^k) \|^2 \notag \\
   & \leq 2\| \nabla F(\Hprm^k)\|^2 + 2L^2 \| \prm_i^k - \Hprm^k\|^2
\end{align}
Together with \eqref{eq: recursive of hat F and consensus}, we observe that 
\begin{equation}
\label{appendix eq: stationary measure recursion}
    \gamma_k  \frac{1}{n} \sum_{i=1}^n \| \nabla F(\prm_i^k) \|^2 \leq 12(L_k - L_{k+1}) + \left(\frac{160L^2(C_{\pi 1}(C_{\pi 2}+C_0\frac{1}{1-\lambda}))}{n^2\rho^2_{\bm A}} \right)\varsigma^2 \gamma_k^3 + \mathcal{O}(\gamma_k\lambda^{2k})
\end{equation}
Summing from $k=0$ to $k=T-1$ for both sides of the equation \eqref{appendix eq: stationary measure recursion} yields
\begin{equation}
    \sum_{k=0}^{T-1} \gamma_k  \frac{1}{n} \sum_{i=1}^n \| \nabla F(\prm_i^k) \|^2 \leq 12(L_0 - L_{T}) + \left(\frac{160L^2(C_{\pi 1}(C_{\pi 2}+C_0 \frac{1}{1-\lambda}))}{n^2\rho^2_{\bm A}} \right)\varsigma^2 \sum_{k=0}^{T-1} \gamma_k^3 + \sum_{k=0}^{T-1}\mathcal{O}(\gamma_k\lambda^{2k})
\end{equation}
Note that $\sum_{k=0}^{T-1}\mathcal{O}(\gamma_k\lambda^{2k}) = {\cal O}( \gamma_T )$, dividing the inequality by $\sum_{k=0}^{T-1}\gamma_k$ and taking the minimum yields the desired result:
\[
{
\min_{k=0,...,T-1}\frac{1}{n}\sum_{i=1}^{n}||\grd F(\prm_{i}^{k})||^{2} = \mathcal{O}\left(\frac{F(\hat{\prm}^{0})-F(\hat{\prm}^{T})+\frac{\varsigma^{2}C_{\pi1}C_{\pi2}}{n^{2}\rho_{\bm A}^{2}}\sum_{k=0}^{T-1}\gamma_{k}^{3}}{\sum_{k=0}^{T-1}\gamma_{k}}\right)
} \]


\section{Proof of Eq.~(8)}
We recall that:
\begin{equation*}
    L_k^c = F(\Hprm^k) + \frac{3\gamma_kL^2}{n\rho_{\bm A}} \| \Prm^k - \HPrm^k \|^2_F 
\end{equation*}
which yields
\begin{equation}
    L^c_{k+1} - L^c_k = F(\Hprm^{k+1}) - F(\Hprm^k) + \frac{3\gamma_kL^2}{n\rho_{\bm A}} \left( \| \Prm^{k+1} - \HPrm^{k+1} \|^2_F -\| \Prm^k - \HPrm^k \|^2_F\right) .
\end{equation}
Applying the recursive relation of $F(\Hprm^k)$ from \eqref{eq:recur_F_new} leads to
\begin{equation}
    L^c_{k+1}  \leq L^c_k - \frac{\gamma_k}{4} \| \grd F( \Hprm^k ) \|^2 - \frac{3 \gamma_k L^2 (2-\rho)}{n \, \rho} \| \Prm^k - \HPrm^k \|_F^2 \notag 
 \quad + \frac{3 \gamma_k L^2}{n \, \rho_{\bm A}} \| \Prm^{k+1} - \HPrm^{k+1} \|_F^2 + {\cal O}(  \gamma_k \lambda^{2k} ).
\end{equation}
Noting that ${\cal O}(  \gamma_k \lambda^{2k} ) = {\cal O}(\gamma_k^3)$ leads to the conclusion.

\end{document}